\documentclass[11pt,reqno]{amsart}
\usepackage[shorthands=off,english]{babel}
\usepackage[T1]{fontenc}
\usepackage[utf8]{inputenc}

\usepackage{amsmath,amsfonts,amssymb}
\usepackage{graphicx} % no decia pdftex
\usepackage{stmaryrd}
\usepackage{float}
\usepackage{tikz}
\usepackage{tikz-cd}
\usepackage[all]{xy}
\usepackage[]{color}
\usepackage{xcolor}
\usepackage{textcomp}
\usepackage{latexsym}
\usepackage{xspace}
\usepackage{mathrsfs}
\usepackage[babel=false]{csquotes}
\graphicspath{ {images/} }
\MakeOuterQuote{"}
\usepackage{fancyvrb}
\usepackage{tabu}  
\usepackage{enumerate}

\usepackage{appendix}
\usepackage{chngcntr}
\usepackage{etoolbox}
\usepackage{lipsum}
\usepackage[formats]{listings}
\usepackage{tikz}
\DeclareMathAlphabet{\mathcal}{OMS}{cmsy}{m}{n} % conservamos el mathcal original

\usepackage{longtable}

\definecolor{mygreen}{rgb}{0,0.6,0}
\definecolor{mygray}{rgb}{0.5,0.5,0.5}
\definecolor{mypurple}{rgb}{0.5, 0.0, 1.0}
\usepackage[hidelinks,hypertexnames=false]{hyperref} %(2)
\hypersetup{
	colorlinks   = true, %Colours links instead of ugly boxes
	urlcolor     = blue, %Colour for external hyperlinks
	linkcolor    = mypurple, %Colour of internal links
	citecolor   = green %Colour of citations
}

\usepackage{amsthm}
\usepackage[english]{cleveref}

\theoremstyle{plain}
\newtheorem{theorem}{\bf Theorem}[section]
\newtheorem{prop}[theorem]{\bf Proposition}
\newtheorem{lemma}[theorem]{\bf Lemma}

\theoremstyle{definition}

\newtheorem{definition}[theorem]{\rm Definition}
\newtheorem{remark}[theorem]{\rm Remark}

\numberwithin{equation}{section}

\crefname{theorem}{}{Theorems}
\crefname{lemma}{}{Lemmas}
\crefname{proposition}{Proposition}{Propositions}
\crefname{corollary}{Corollary}{Corollaries}
\crefname{definition}{}{Definitions}
\crefname{remark}{}{Remarks}
\crefname{example}{Example}{Examples}
\crefname{question}{Question}{Questions}
\crefname{conjecture}{Conjecture}{Conjectures}

\begin{document}
\title[On $h_0b_n\gamma_s$]{The composition of R.~Cohen's elements and the third periodic elements in stable homotopy
groups of spheres}
\author[X.~Gu, X.~Wang and J.~Wu]{Xing Gu, Xiangjun Wang and Jianqiu Wu*}


\thanks{ *Corresponding author}
\thanks{ Project supported by the National Science Foundation of China (No.11871284 and No.11761072)}
\maketitle

\begin{abstract}
 In this paper, we re-compute the cohomology of the Morava stabilizer algebra $S(3)$ \cite{Ra1, Ya}. As an application,
 we show that for $p \geq 7$, if  $s\not \equiv 0, \pm 1 \,\,  mod \,(p) $, $n\not \equiv 1 \,\, mod\, 3$, $n>1$,
then  $\zeta_n\gamma_s$ is a nontrivial product in $\pi_*(S)$ by Adams-Novikov spectral sequence, where $\zeta_n$ is
created by  R. Cohen \cite{Co}, $\gamma_s$ is a third periodic homotopy elements.
\end{abstract}

\section{Introduction}
  In this paper we adapt the well-known framework of classical Adams spectral sequence, Adams-Novikov spectral sequence and chromatic spectral sequence, as described in \cite{Ra}. Fix $p$  an odd prime. Consider the corresponding Brown-Peterson spectrum $BP$, of which the coefficient ring $\pi_{*}(BP)$ is denoted by $BP_*$, and the $BP$-homology of the spectrum $BP$ is denoted by $BP_*BP$. There is a well-known Hopf algebroid structure over the pair $(BP_*,BP_*BP)$.

Let $K(n)_*$ be the coefficient ring of the $n$-th Morava K-theory, $\Sigma(n) = K(n)_* \otimes_{BP_*} BP_*BP \otimes_{BP_*} K(n)_* $, and  $S(n)=\Sigma(n)\otimes_{K(n)_*}\mathbb{Z}/p$ be the $n$-th Morava stabilizer algebra. $\Sigma(n)$ and $S(n)$ have obvious coproducts induced by that of $BP_*BP$, making them Hopf algebras.

At an odd prime $p\geqslant 5$, the cohomology of Hopf algebra $S(3)$ has been studied
by  Ravenel in \cite{Ra1}, where he gave the Poincare series
of $H^*S(3)$ and listed the generators bellow dimensional 5.
It is also studied by  Yamaguchi in \cite{Ya},  where he shown the ring structure, though there may be some misprints.

In this paper,  we redetermine the $\mathbb{Z}/p$-algebra structure of $H^{*}S(3)$, i.e., the $\mathbb{Z}/p$-algebra
$$\operatorname{Ext}_{S(3)}(\mathbb{Z}/p,\mathbb{Z}/p)$$
for $p\geq 7$ in another way, and apply this result to detect a nontrivial product in the stable homotopy groups of spheres.
It will become clear that the algebra structure, rather than the underlying $\mathbb{Z}/p$-module structure of
$H^*S(3)$, is essential to our application.

We define a May-type filtration upon $S(3)$ in such a way that $E^{*,*}(3)=\bigoplus_{M\geqslant 0}F^{*,M}S(3)/F^{*,M-1}S(3)$
becomes a primitive generated Hopf algebra. This filtration gives rise to a May spectral sequence
$\{E_r^{s,t,M}, d_r\}$ that converges to $H^*S(3)$.  A simple argument in homological algebra then determines $E_{1}^{*,*,*}$ and $d_1$.
Hence, $E_{2}^{*,*,*}$ for $H^*S(3)$ can be obtained by direct computation. Finally, a comparison of $E_{2}^{*,*,*}$ and the cobar complex
of $S(3)$ gives us the desired result. The structure of $H^*S(3)$ is rather complicated and therefore postponed to Section 3.

We apply the result above to detect a family of nontrivial elements in the homotopy group of the sphere spectrum,
each of which is the product of following two well-known elements.
To describe the first one, recall
$$ BP_*=BP_*S=\mathbb{Z}_{(p)}[v_1, v_2, \cdots],$$
where $v_i$ is the i-th Hazewinkel generator with degree $2(p^i-1)$(\cite{Ha,Ha1}\cite{Ra}). We recall the Greek letter elements in the Adams-Novikov $E_2$ page $H^*(BP_*BP)=\operatorname{Ext}_{BP_*BP}(BP_*,BP_*)$.

Consider the short exact sequence of graded $\mathbb{Z}_{(p)}$-modules
$$0\rightarrow BP_*/I_n\xrightarrow{v_n}BP_*/I_n\rightarrow BP_*/I_{n+1}\rightarrow 0,$$
where $I_{n+1}=(p, v_1,\cdots, v_n)$, the ideal generated by $p, v_1, \dots, v_n$. By convention, we also set $v_0=p$ and $I_{-1}=0$. Furthermore, we let
$$\delta_n: \operatorname{Ext}^s(BP_*/I_{n+1})\rightarrow  \operatorname{Ext}^{s+1}(BP_*/I_{n})$$
denote the connecting homomorphism corresponding to the short exact sequence above, and for $t, n>0$, let
$$\alpha_t^{(n)}=\delta_0\delta_1\cdots\delta_{n-1}(v_n^t)\in \operatorname{Ext}^n(BP_*).$$
Conventionally we denote $\alpha_t^{(n)}$ for $n=1,2,3$ by $\alpha_t, \beta_t, \gamma_t$, respectively. Toda(\cite{To}, \cite{Ra}) proved the following
\begin{theorem}[\cite{Ra}, Theorem 1.3.18 (b)]\label{gamma}
 For $p \geq 7$, each $\gamma_t$ is represented by a nontrivial element of order $p$ in $\pi_{tq(p^2+p+1)-q(p+2)-3}(S^0)$, where
$q=2p-2$.
\end{theorem}
From now on we consider $\gamma_t$ also as the element in $\pi_*(S^0)$ that it represents. This is the first factor of the product that concerns us.

The other factor is first detected with the classical Adams spectral sequence by Cohen (\cite{Co}), which he denotes by $\zeta_n$, a permanant cocycly of
bi-degree $(3, 2(p-1)(1+p^{n+1}))$.

Our structure theorem of $H^*S(3)$ leads to the following
\begin{theorem}\label{nontrivprod}
 For $p \geq 7$, $s\not \equiv 0, \pm 1 \,\,  mod \,(p) $, if $n\not \equiv 1 \,\, mod\, (3)$, $n>1$  then $0 \neq \zeta_n \gamma_s \in \pi_*(S)$.
\end{theorem}

We briefly explain the idea of the proof of Theorem \cref{nontrivprod}.
Let  $\mathcal{A}_*$ be the dual of the mod $p$ Steenrod algebra and consider the mod $p$ Thom map
$$\Phi: BP\longrightarrow K(\mathbb{Z}/p),$$
where the latter is the Eilenberg-Maclane Spectrum associated to $\mathbb{Z}/p$. This map induces a homomorphism
$$\Phi: \operatorname{Ext}_{BP_*BP}(BP_*,BP_*)\longrightarrow\operatorname{Ext}_{\mathcal{A}_*}(\mathbb{Z}/p,\mathbb{Z}/p).$$
Cohen \cite{Co} detected that $h_0b_n\in \operatorname{Ext}_{\mathcal{A}_*}(\mathbb{Z}/p,\mathbb{Z}/p)$, $n>0$, is a permanent cycle in the
classical Adams spectral sequence and it converges to $\zeta_n$ in $\pi_*S^0$. From the Thom map we find that it was
$\alpha_1(\beta_{p^n/p^n}+x)\in\operatorname{Ext}_{BP_*BP}(BP_*,BP_*)$ that converges to $\zeta_n$ in the
Adams-Novikov spectral sequence,   where $ x = \sum\limits_{s,k, j } a_{s, k, j}\beta_{sp^k/j}$ and $0 \leq a_{s, k, j} \leq p-1$, $a_{1, n, p^n} =0$.

   Consider the canonical homomorphism $BP_*\longrightarrow v_3^{-1}BP_*/I_3$, which induces homomorphism
\[
Ext^{*,*}_{BP_*BP}(BP_*, BP_*)\longrightarrow Ext^{*,*}_{BP_*BP}(BP_*, v_3^{-1}BP_*/I_3).
\]
On the other hand, consider the map $BP_*\rightarrow K(3)_*$, where $K(3)$ is the $3$-rd Morava K-theory.
By the change of ring theroem in Chapter 5 of \cite{Ra}, we have
$$\operatorname{Ext}_{BP_*BP}(BP_*, v_3^{-1}BP_*/I_3)\cong \operatorname{Ext}_{K(3)_*K(3)}(K(3)_*,K(3)_*)
=H^*S(3)\otimes\mathbb{Z}/p[v_3, v_3^{-1}]$$
and $\varphi$ is the composition
\[
\varphi: Ext^{*,*}_{BP_*BP}(BP_*, BP_*)\longrightarrow Ext^{*,*}_{BP_*BP}(BP_*, v_3^{-1}BP_*/I_3)
\cong H^*S(3)\otimes\mathbb{Z}/p[v_3, v_3^{-1}].
\]

 We find the images of the representation of
$\alpha_1(\beta_{p^n/p^n}+  x)$ and $\gamma_s$ under $\varphi$
and show that the product of their images reduction in $H^*S(3)$ is nontrivial. This implies that
$\alpha_1(\beta_{p^n/p^n}+  x)\gamma_s$  is nontrivial in $\operatorname{Ext}_{BP_*BP}(BP_*,BP_*)$
and then $\zeta_n\gamma_s$ is nontrivial in $\pi_*(S)$.
This is why the algebra structure of $H^*S(3)$ is essential.

This paper is organized as follows. In section 2, we define a May-type filtration upon the Hopf algebra $S(3)$
and consider the corresponding spectral sequence $\{E_r^{s,t,M}, d_r\}\Longrightarrow H^*S(3)$.
In section 3, we recompute the cohomology ring of the Morava stabilizer algebra $S(3)$ with the spectral sequence constructed in section 2. In section 4, we prove that the product $\zeta_n\gamma_s\in\pi_*(S^0)$ is nontrivial.

\section{ The May Spectral Sequence for $H^*(S(3))$}

\subsection{ The May spectral sequence}

Let $p $ be a prime, $ BP_*=\mathbb{Z}_{(p)}[v_1, v_2, \cdots] $ and $BP_*BP=BP_*[t_1, t_2, \cdots]$. For the Hazewinkel's generators described inductively by
$ v_s=pm_s - \sum\limits_{i=1}^{s-1}v_{s-i}^{p^i}m_i$ (cf \cite{Ha,Mi,Ra}). The coproduct map $\Delta: BP_*BP \to BP_*BP \otimes_{BP_*} BP_*BP$ is given by
$$ \sum\limits_{i+j=s}m_i(\Delta t_j)^{p^i}=\sum\limits_{i+j+k=s}m_it_j^{p^i}\otimes t_k^{p^{i+j}} $$
and the right unit $\eta_R : BP_* \to BP_*BP$ is given by
$$ \eta_R(m_n) = \sum\limits_{i+j=n}m_it_j^{p^i}. $$
One can easily prove that
\begin{eqnarray} \Delta(t_1)=t_1 \otimes 1 +1 \otimes t_1\end{eqnarray}
and
\begin{eqnarray} \Delta(t_2)= \sum\limits_{i+j=2}t_i \otimes t_j^{p^i}-v_1 b_{1,0} \end{eqnarray}
where $p \cdot b_{1,0}=\Delta(t_1^p)-t_1^p \otimes 1 - 1 \otimes t_1^p$. Inductively define
$$p \cdot b_{s,k-1}=\Delta (t_s^{p^k})- \sum\limits_{i+j=s}t_i^{p^k}\otimes t_j^{p^{i+k}}+ \sum\limits_{0 < i < s}v_i^{p^k}b_{s-i,k+i-1},$$
 one has
$$\Delta (t_{s+1})= \sum\limits_{i+j=s+1}t_i\otimes t_j^{p^i}- \sum\limits_{i=1}^s v_ib_{s+1-i,i-1}.$$
It is convenient to give some specific examples, which can be found in \cite{Ka,Lee} :
\begin{eqnarray}
 \eta_R(v_1) &=& v_1+pt_1   \notag\\
 \eta_R(v_2) &\equiv& v_2 + v_1t_1^p+pt_2 -v_1^pt_1 \quad mod \, (p^2,v_1^{p^2}) \notag\\
 \eta_R(v_3) &\equiv& v_3+v_2t_1^{p^2}+v_1t_2^p+pt_3-v_2^pt_1 -v_1^2v_2^{p-1}t_1^p \quad mod \, (p^2, pv_1, v_1^3) \notag\\
 \Delta(t_5) &\equiv& t_5 \otimes 1 +1 \otimes t_5+t_4 \otimes t_1^{p^4}+t_3 \otimes t_2^{p^3}+t_2 \otimes t_3^{p^2}+t_1 \otimes t_4^p -v_3b_{2, 2}-v_4b_{1, 3}\quad mod \, (p, v_1, v_2)  \notag\\
\end{eqnarray}
where

$b_{1,k}= \sum\limits_{i=1}^{p^{k+1}-1} $ $p^{k+1}\choose i$$ /p$ $ t_1^i \otimes t_1^{p^{k+1}-i}$,

$b_{2,k}= \frac{1}{p} \left(\Delta (t_2^{p^{k+1}})- \sum\limits_{i+j=2}t_i^{p^{k+1}}\otimes t_j^{p^{i+{k+1}}}+ v_1^{p^{k+1}}b_{1,k+1} \right )$.

Thus, for the Morava K-theory $K(3)_*=\mathbb{Z}/p[v_3, v_3^{-1}]$, the Hopf algebra $\Sigma(3)=K(3)_* \otimes_{BP_*} BP_*BP \otimes_{BP_*} K(3)_*$ is isomorphic to
$$\Sigma(3)=K(3)_*[t_1, t_2,...]/(v_3t_i^{p^3}-v_3^{p^i}t_i), \quad for \, i \geq 1.$$
And $S(3)=\mathbb{Z}/p \otimes_{K(3)_*}\Sigma(3) \otimes_{K(3)_*}  \mathbb{Z}/p$ is isomorphic to
$$ S(3)=\mathbb{Z}/p[t_1, t_2,...]/(t_i^{p^3}-t_i), \quad for \, i \geq 1.$$
The inner degree of $t_s$ in $S(3)$ is
$$ |t_s| \equiv 2(p^s-1) \quad mod \, 2(p^3-1),$$
because $v_3$ is sent to $1$. The structure map $\Delta: S(3) \to S(3) \otimes S(3)$ acts on $t_s$ as follows
\begin{equation}
 \Delta(t_s)=
\begin{cases}
 t_s \otimes 1 +1 \otimes t_s +\sum\limits_{1 \leq k \leq s-1}t_k \otimes t_{s-k}^{p^k}& \text{if $s\leq 3 $ },\\
 t_s \otimes 1 +1 \otimes t_s +\sum\limits_{1 \leq k \leq s-1}t_k \otimes t_{s-k}^{p^k}-b_{s-3,2}& \text{if $s> 3$}.
\end{cases}
\end{equation}
Here $  b_{1,0}=\frac{1}{p}(\Delta(t_1^p)-t_1^p \otimes 1 - 1 \otimes t_1^p)$ and
$$ b_{s,k-1}=\frac{1}{p}\left(\Delta (t_s^{p^k})- \sum\limits_{i+j=s}t_i^{p^k}\otimes t_j^{p^{i+k}}+ b_{s-3,k+2}\right).$$

\begin{definition} \label{MayFLT}
 In the Hopf algebra S(3), we define May filtration M as follows:
 \begin{enumerate}
 \item For $s=1,2,3$, set the May filtration of $t_s^{p^j}$ as $M(t_s^{p^j})=2s-1$.
 \item For $s>3$ and $j \in \mathbb{Z}/3$, from $M(b_{s-3,j})=p\cdot M(t_{s-3}^{p^j})$,
inductively set  the May filtration of $t_s^{p^j}$ as
       $$ M(t_s^{p^j})=max\left\{M(t_k^{p^j})+M(t_{s-k}^{p^{j+k}}),\ \ \   p\cdot M(t_{s-3}^{p^{j+2}})|0< k < s\right\}+1 . $$
\end{enumerate}
\end{definition}

%\begin{lemma}
%   Let $s_0= max\{ [\frac{2pn+p-2}{2(p-1)}], n\}$. Then the May filtration of $t_s^{p^j}$ satisfies
%   \begin{equation}
% M(t_s^{p^j})=
%\begin{cases}
% 2s-1& \text{if $s \leq s_0$},\\
%p \, M(t_{s-n}^{p^j})+1& \text{if $s> s_0$}
%\end{cases}
%\end{equation}
%and $ M(t_s^{p^j}) > M(t_i)+M(t_{s-i})$.
%\end{lemma}

%  \begin{proof}
%   From the definition above, we see that for $s \leq n$, the May filtration of $t_s^{p^j}$ is $2s-1$. If $s \leq [ \frac{2pn+p-2}{2(p-1)}]$ then  $p\,(2(s-n)-1)+1 \leq 2s-1$ and from $s-n \leq n$ we see that $ M(t_{s-n}^{p^j})=2(s-n)-1$, thus the May filtration of $t_s^{p^j}$ is $2s-1$. $M(t_i)+M(t_{s-i})=2i-2 < 2i-1 =M(t_s^{p^j})$

%   Since $s_0 +1 > \frac{2pn+p-2}{2(p-1)}$, we see that $p(2(s_0+1-n)-1)+1 > 2(s_0+1)+1$ and the May filtration of $t_{s_0+1-n}^{p^j}$ is $2(s_0 +1-n)-1$. Thus $M(t_{s_0+1}^{p^j})=pM(t_{s_0+1-n}^{p^j})+ 1$, and for $0 <i < s_0+1$, $M(t_i)+M(t_{s_0+1}-i)=2(s_0+1)-2 <M(t_{s_0+1}^{p^j})$. Therefore, the lemma can be proved by induction.
%    \end{proof}

%Let $F^{*, M}(n)$ be the sub-module of S(n) generated by the elements with the May filtration no larger than M, then it is a Hopf algebra.

%Set $E^{*, M}(n)=F^{*, M}(n)/ F^{*, M-1}(n)$. One can see from Lemma 2.2 that

Let $F^{*, M}S(3)$ be the sub-module of $S(3)$ generated by the elements with May filtration no larger than $M$. Set
$E^{*, M}(3) = F^{*, M}S(3)/F^{*, M-1}S(3)$. From (2.4), we get the coproduct in
\[
E^{*, *}(3)=\bigoplus_{M\geqslant 0}F^{*,M}S(3)/F^{*,M-1}S(3),
\]
that is, $\Delta(t_s)= t_s \otimes 1 +1 \otimes t_s.$ Thus
   \begin{equation}
     E^{*, *}(3)\cong \bigotimes_{s\geqslant 1} T[t_s^{p^j}| j\in \mathbb{Z}/3],
   \end{equation}
   is a primitively generated Hopf algebra, where $T[\,]$ denote the truncated polynomial algebra of height p on the indicated generators, and each $t_s^{p^j}$ is a primitive element.

   Let $C^{s, *}S(3)= S(3)^{\otimes s}$ denote the cobar construction of $S(3)$. The differential $d: C^{s, t}S(3) \to C^{s+1, t}S(3)$ is given on generators as
    \begin{eqnarray}
  d( \alpha_1 \otimes \cdots  \otimes \alpha_s )=\sum\limits_{i=1}^s(-1)^i \alpha_1  \otimes \cdots \otimes \alpha_{i-1} \otimes \Delta(\alpha_i) \otimes \alpha_{i+1}\cdots  \otimes \alpha_s  \\
       + 1 \otimes \alpha_1 \otimes \cdots  \otimes \alpha_s + (-1)^{s+1} \alpha_1 \otimes \cdots  \otimes \alpha_s \otimes 1. \notag
  \end{eqnarray}
  In general, the generator $\alpha_1 \otimes \cdots  \otimes \alpha_s $ of $C^{s, t}S(3)$ is denoted by $[\alpha_1 | \cdots  | \alpha_s]$. For the generator $[\alpha_1 | \cdots  | \alpha_s]$, define its May filtration as
  $$M([\alpha_1 | \cdots  | \alpha_s])=M(\alpha_1)+ \cdots + M(\alpha_s). $$

  Let $F^{*, *, M}$ denote the sub-complex of $C^{*, *}S(3)$ generated by the elements with May filtration not greater than M. Then we obtain  a short exact sequence
  \begin{equation}
     0 \to  F^{*, *, M-1} \to  F^{*, *, M} \to  E_0^{*, *, M} \to  0
   \end{equation}
   of cochain complexes, where $E_0^{*, *, M}$ denote $ F^{*, *, M}/F^{*, *, M-1}$. The cochain complex
$E_0^{*, *, *}$ is isomorphic to the cobar complex of $E^{*, *}(3)$
given in (2.5). Let $E_1^{s, *, M}$ be the homology of $(E_0^{*, *, M}, d_0)$. Then (2.7)
gives rise to the May spectral sequence $\{ E_r^{s, t, M}S(3), d_r \}$ that converges to
$H^{s,t}S(3)\stackrel{\Delta}{=}H^{s, t}(C^{*, t}S(3), d)= \operatorname{Ext}_{S(3)}^{s, t}(\mathbb{Z}/p, \mathbb{Z}/p)$
as $\mathbb{Z}/p$-algebras.

   \begin{theorem}
    The Hopf algebra $S(3)$ can be given an increasing filtration as in definition \cref{MayFLT}.
The associated spectral sequence, so called May spectral sequence (MSS) converges to
$H^*S(3)$. The $E_1$-term $E_1^{s, t, M}$ is isomorphic to
    $$ E[h_{i,j}|i \geq 1, j \in \mathbb{Z}/3] \otimes P[b_{i,j}| i \geq 1, j \in \mathbb{Z}/3].$$
   The homological dimension of each element is given by $s(h_{i, j})=1$,$s(b_{i, j})=2$ and the degree is given by
   $$h_{i,j} \in E_1^{1,2(p^i-1)p^j, *}S(3), \quad b_{i,j} \in E_1^{2,2(p^i-1)p^{j+1,*}}S(3),$$
   here $h_{i, j}$ corresponds to $t_i^{p^j}$ and $b_{i, j}$ corresponds to $\sum\limits_{k=1}^{ p-1}\binom{p}{k}/p \, [t_i^{kp^j}|\,t_i^{(p-k)p^j}]$.
   One has $d_r  : E_r^{s, t, M}S(3) \to E_r^{s+1, t, M-r}S(3)$. If $x \in E_r^{s, t, M}$, then $$d_r(xy)=d(x)\cdot y+(-1)^s x \cdot d_r(y).$$

In the $E_1$-term of this spectral sequence, we have the following relations:
    \begin{center}
        $h_{i, j}\cdot h_{i_1,j_1}=- h_{i_1, j_1} \cdot h_{i, j}$,

        $h_{i, j} \cdot b_{i_1, j_1}=b_{i_1, j_1} \cdot h_{i, j}$,

        $b_{i, j} \cdot b_{i_1, j_1}=b_{i_1, j_1} \cdot b_{i, j}$ .
    \end{center}
 \end{theorem}

\begin{proof}
 From \cite{Ra}, we can see that for the truncated polynomial algebra $T[x]$ with $|x|\equiv 0$ mod 2 and $x$ primitive,
 $$\operatorname{Ext}_{T[x]}(\mathbb{Z}/p, \mathbb{Z}/p)=E(h) \otimes P(b)$$
where $h \in \operatorname{Ext}^1$ is represented by $[x]$ in the cobar complex and $b \in \operatorname{Ext}^2$ by
$\sum\limits_{i=1}^{p-1} \binom{p}{i}/p [x^i|x^{p-i}]$. Notice that the $E_0$-term of the spectral sequence is isomorphic
to the cobar complex of $E^{*, *}(3)$ given by (2.5), we see that
$$ H^{s, *, M}(E_0^{*, t, M}, d_0)=\operatorname{Ext}_{E^{\ast, \ast}(3)}^{s, t}(\mathbb{Z}/p, \mathbb{Z}/p)=
\bigotimes_{s\geqslant 1} \operatorname{Ext}_{T[t_s^{p^j}|j \in \mathbb{Z}/3]}^{*, *}(\mathbb{Z}/p, \mathbb{Z}/p).$$
Thus, the May's $E_1$-term
$$E_1^{s, t, M}=E[h_{i,j}|i \geq 1, j \in \mathbb{Z}/3] \otimes P[b_{i,j}|i \geq 1, j \in \mathbb{Z}/3].$$
Notice that $d_0(t_i^{p^j} \cdot t_{i_1}^{p^{j_1}})=-t_i^{p^j} \otimes t_{i_1}^{p^{j_1}} - t_{i_1}^{p^{j_1}} \otimes t_i^{p^j}$, we get $h_{i, j}\cdot h_{i_1,j_1}=- h_{i_1, j_1} \cdot h_{i, j}$. In a similar way, one can prove that   $h_{i, j} \cdot b_{i_1, j_1}=b_{i_1, j_1} \cdot h_{i, j}$ and
        $b_{i, j} \cdot b_{i_1, j_1}=b_{i_1, j_1} \cdot b_{i, j}$ .
\end{proof}

\subsection{ The first May differential}

From now on we fix  $p\geqslant 7$ being an odd prime. From Definition \cref{MayFLT}, we see that the May filtration
is given by
\begin{align*}
M(t_1^{p^j}) = & 1, & M(t_2^{p^j}) = & 3, & M(t_3^{p^j}) = & 5, \\
M(t_4^{p^j}) = & p+1, & M(t_5^{p^j})= & 3p+1, & M(t_6^{p^j}) = & 5p+1.
\end{align*}
By induction we see that for $s=1,2,3$
\[
M(t_{3r+s}^{p^j})=p\cdot M(t_{3r+s-3}^{p^{j+2}})+1=(2s-1)p^r+p^{r-1}+\cdots +1>M(t_{k}^{p^j})+M(t_{3r+s-k}^{p^{j+k}})+1,
\]
here $0<k<3r+s$. Thus from (2.4) one has the first May differential $d_1: E_1^{s,*,M}S(3)\longrightarrow E_1^{s+1,*, M-1}S(3)$
\begin{equation}
 d_1(h_{s,j})=
\begin{cases}
 -\sum\limits_{i=1}^{s-1}h_{i, j}\,h_{s-i,j+i}& \text{if $s \leq 3$},\\
b_{s-3, j+2}& \text{if $s> 3$.}
\end{cases}
\end{equation}
Each  $b_{s, j}$ is the boundary of the first May differentials.

\begin{theorem}
The  $E_2$-term of the May spectral sequence is isomorphic to the cohomology  of
\begin{eqnarray*}
E[h_{3, j},h_{2, j}, h_{1, j}| j\in \mathbb{Z}/3].
\end{eqnarray*}
The first May differential is given by
\begin{eqnarray*}
 d_1(h_{s,j}) &=&-\sum\limits_{i=1}^{s-1}h_{i, j}\,h_{s-i,j+i}  \quad\quad \,\, \, \mbox{ for $s \leq 3$. }
\end{eqnarray*}
\end{theorem}

\begin{proof}
   From the May's $E_1$-term
  we define a filtration, for each $n \geqslant 1$
   \begin{equation}
         F(n)=
        \begin{cases}
           E[h_{i, j}|1 \leq i \leq n, j \in \mathbb{Z}/3]& \text{for $ 1 \leq n \leq 3$},\\
           E[h_{i, j}|1\leq i \leq n, j \in \mathbb{Z}/3]\otimes P[b_{i, j}|1 \leq i\leq n-3, j \in \mathbb{Z}/3]& \text{for $n>3$}.\\
        \end{cases}
   \end{equation}
   The filtration gives rise to a spectral sequence and thus gives the theorem.
\end{proof}

To compute the $E_2$-page of the May spectral sequence
\[
E_2^{s,*,M}=H^{s,*,M}(E[h_{3,j}, h_{2,j}, h_{1,j}|j\in \mathbb{Z}/3]),
\]
we will give a filtration on the exterior algebra $F(n)=E[h_{i,j}|1\leqslant i\leqslant n, j\in\mathbb{Z}/3]$ for
$n=2,3$. This filtration gives rise to a spectral sequence and the spectral sequences allow
us to compute $H^*(F(2))$ from $H^*(F(1))$ and then compute $H^*(F(3))$ from $H^*(F(2))$ (cf \cite{Ra1}).

Let $E^i(n)=\mathbb{Z}/p[h_{n, j_1}\cdots h_{n, j_i}]$, the sub-module generated by elements of homological dimension
i,
and $h_{n, j_k} \neq h_{n, j_l}$ if $j_k \neq j_l$. Then in $F(n)$ for $n=2,3$, let
\begin{equation}
F^{k}(n)=\bigoplus\limits_{i \leq k} E^i(n) \otimes E[h_{i, j}|1 \leq i \leq n-1, j\in \mathbb{Z}/3],
\end{equation}
then we have the following statement.
\begin{theorem}[\cite{Ra1}, (1.10) Theorem]\label{DFS(3)}
The spectral sequence induced by  the filtration (2.8) converges to the cohomology of
$F(n)=E[h_{i, j}|1 \leq i \leq n, j\in \mathbb{Z}/3]$, and its  $E_1$-term can be described as
$$ \widetilde{E_1}^{\ast, \ast, \ast,  *}(n)=E[h_{n, j}|j\in \mathbb{Z}/3]\otimes
 H^{\ast}E[h_{i, j}|1 \leq i \leq n-1, j\in \mathbb{Z}/3].$$
The differential is given by
$$\delta_r: \widetilde{E}^{s, t, M, k}_{r}(n) \longrightarrow \widetilde{E}^{s+1, t, M-1, k-r}_{r}(n),$$
 and the first differential is expressed as
$$\delta_1(h_{n,j_1}h_{n, j_2}\cdots  h_{n, j_{k}}x)= \sum\limits_{i=1}^{k}(-1)^{i-1}h_{n, j_1}h_{n, j_2} \cdots d(h_{n, j_i})\cdots h_{n, j_{k}}x,$$
where $x$ is a cohomology class in $H^*E[h_{i, j}|1 \leq i \leq n-1, j\in \mathbb{Z}/3]$.
\end{theorem}

\section{The cohomology ring of Morava stabilizer algebra S(3)}

In this section we recompute the cohomology of $S(3)$ at prime $p \geq 7$, with the add of the May filtration given in definition \cref{MayFLT}. First we consider the
differential graded algebra
$$ F(3)= E[h_{i, j}| 1  \leq i \leq 3, j \in \mathbb{Z}/3], $$
whose differentials defined by
\begin{equation}
 d_1(h_{i, j})= - \sum\limits_{1 \leq k \leq i}h_{k, j}h_{i-k, j+k}
\end{equation}
and
$$ d_1(xy)=d_1(x)y +(-1)^{s}xd_1(y) $$
for any monomials $x$, $y$ and $s$ denotes the homological dimension of $x$.
To calculate the cohomology of   $F(3)= E[h_{i, j}| 1  \leq i \leq 3, j \in \mathbb{Z}/3]$,
we will inductively calculate the cohomology of $F(n)$ for $n=1, 2, 3$ as it
is indicated by Theorem \cref{DFS(3)}.

First notice that
$H_*E[h_{1,i}|i\in\mathbb{Z}/3]=E[h_{1,i}|i\in\mathbb{Z}/3]$
and we have the spectral sequence
\[
\widetilde{E}_1^{*,*,*,*}(2)=E[h_{2,i}|i\in\mathbb{Z}/3]\otimes H_*E[h_{1,i}]\Rightarrow H_*E[h_{2,i}, h_{1,i}|i\in\mathbb{Z}/3].
\]

From this spectral sequence one can easily get the generators of $H^*E[h_{2, i}, h_{1, i} |i \in \mathbb{Z}/3]$, they
are listed as follows:
\begin{enumerate}[\notag]
\item  dim 0: 1;
\item  dim 1: $h_{1, i}$;
\item  dim 2: $g_i\triangleq  h_{2,i}h_{1,i}$, \quad $k_i \triangleq  h_{2,i}h_{1,i+1}$, \quad  $e_{3,i} \triangleq  h_{1, i}h_{2,i+1}+h_{2, i}h_{1,i+2}$,
            \quad $\left(\sum_i e_{3, i}=  0\right)$ ;
\item  dim 3: $c_i \triangleq h_{2,i}h_{2, i+1}h_{1, i}+h_{2,i+2}h_{2, i}h_{1, i+1}$, \quad $ h_{2,i}h_{2, i+1}h_{1, i+1}$, \quad $g_ih_{1,i+1}$, \quad $ e_{3, i}h_{1, i}$;
\item  dim 4: $ e_{3, i+1}g_{i}$,\quad $ e_{3, i}k_i$, \quad $e_{3,i}^2$, \quad  $\left(\sum_i e_{3,i}^2\simeq 0\right)$;
\item  dim 5:  $e_{3, i}c_i$;
\item  dim 6:  $e_{3, i}^2e_{3, i+1} = -2h_{2,i}h_{2,i+1}h_{2,i+2}h_{1,i}h_{1,i+1}h_{1,i+2}$ \quad $(e^2_{3,i}e_{3,i+1}=e_{3, i+1}^2e_{3, i+2})$.
\end{enumerate}
where $i\in\mathbb{Z}/3$. We also list the product relations with $e_{3,i}$ in $H^*E[h_{2,i}, h_{1,i}|i\in \mathbb{Z}/3]$
which will be used in computing $H^*E[h_{3,i}, h_{2,i}, h_{1,i}|i\in\mathbb{Z}/3]$ by the spectral sequence given
in Theorem \cref{DFS(3)}:
\[
\begin{array}{c}
\mbox{Table 3.1  \ \ Product relations with $e_{3,i}$}\\
\begin{tabular}{|l|l|l|l|}
\hline
\multicolumn{1}{|c|}{dimension} &
\multicolumn{1}{c|}{relations}  &
  &  \\
\hline
\hline
    dim 3: & $e_{3, i+1} \cdot  h_{1, i} \simeq e_{3, i}h_{1, i}$, & $e_{3, i+2}\cdot  h_{1, i} \simeq -2e_{3, i}h_{1, i};$ & \\
  & & &\\
    dim 4: & $e_{3, i} \cdot e_{3, i+1} \simeq e_{3, i+2}^2,$ &  $e_{3, i}\cdot g_i =0,$  & $e_{3, i+2}\cdot g_{i} = -e_{3, i+1} g_{i},$  \\
           & $e_{3, i+1} \cdot k_{i} = -e_{3, i}k_i,$ &  $e_{3, i+2} \cdot k_{i} = 0;$ & \\
  & & &\\
    dim 5: & $e_{3, i}\cdot e_{3, i}h_{1, i} =0,$ & $e_{3, i+1}\cdot e_{3, i}h_{1, i} \simeq 0,$  & $e_{3, i+2}\cdot e_{3, i}h_{1, i} =0,$ \\
           & $e_{3, i}\cdot g_ih_{1, i+1} =0,$ &  $e_{3, i+1}\cdot g_ih_{1,i+1} = 0,$  & $e_{3, i+2}\cdot g_ih_{1,i+1} =0,$ \\
           & $e_{3, i}\cdot h_{2,i}h_{2, i+1}h_{1, i+1} =0,$ & $e_{3, i+1}\cdot h_{2,i}h_{2, i+1}h_{1, i+1} = 0,$  & $e_{3, i+2}\cdot h_{2,i}h_{2, i+1}h_{1, i+1} = 0,$\\
           & $e_{3, i+1}\cdot c_{i} =-2e_{3, i}c_{i},$ & $e_{3, i+2}\cdot c_{i} = e_{3, i}c_{i};$ & \\
  & & & \\
    dim 6: & $e_{3, i}\cdot e_{3,i}^2 = 0,$    & $e_{3, i}\cdot e_{3,i+1}^2 = -e_{3, i+1}^2e_{3, i+2},$ & $e_{3, i} \cdot e_{3, i+1}g_{i} = 0,$  \\
           & $e_{3, i+1}\cdot e_{3, i+1}g_{i} = 0,$ & $e_{3, i+2}\cdot e_{3, i+1}g_{i} = 0,$ & $e_{3, i}\cdot e_{3, i}k_i = 0,$ \\
           & $e_{3, i+1}\cdot e_{3, i}k_i = 0,$     & $e_{3, i+2} \cdot e_{3, i}k_i = 0.$ & \\
\hline
\end{tabular}
\end{array}
\]

Now we  calculate  $H^*E[h_{3, i}, h_{2,i}, h_{1,i}|i\in\mathbb{Z}/3]$. From Theorem \cref{DFS(3)} we have the spectral sequence
\[ \widetilde{E}_1^{*,*,*,*}(3)=E[h_{3,i}|i\in\mathbb{Z}/3]\otimes H^*E[h_{2,i}, h_{1, i}]\Rightarrow H^*E[h_{3, i}, h_{2,i}, h_{1,i}|i\in\mathbb{Z}/3], \]
with the first differential
$$\delta_1:  \widetilde{E}_1^{s, t, M, k}(3) \longrightarrow \widetilde{E}_1^{s+1, t, M-1, k-1}(3).$$
To calculate the $E_2$-term, we denote the generators $h_{3,i}h_{3,j}\in E^2[h_{3,i}|i\in\mathbb{Z}/3]$ by $h_{3,i}h_{3,i+1}$, $i\in \mathbb{Z}/3$
and denote $h_{3,0}h_{3,1}h_{3,2}\in E^3[h_{3,i}|i\in\mathbb{Z}/3]$ by $h_{3,i}h_{3,i+1}h_{3,i+2}$,
$h_{3,i}h_{3,i+1}h_{3,i+2}=h_{3,i+1}h_{3,i+2}h_{3,i+3}$. Then
\begin{align}
\delta_1(h_{3,i}) = & -e_{3,i} \notag\\
\delta_1(h_{3,i}h_{3,i+1})  = & -h_{3,i}e_{3,i}  - (h_{3,i+1}e_{3,i} + h_{3,i}e_{3,i+2}) \\
\delta_1(h_{3,i}h_{3,i+1}h_{3,i+2}) = & -\sum_i h_{3,i}h_{3,i+1}e_{3,i+2}, \notag
\end{align}
and the generators in $E_1$-term can be written as one of forms $x$, $h_{3, i}x$, $h_{3, i+1}x$, $h_{3, i+2}x$, $h_{3, i}h_{3, i+1}x$, $h_{3, i+1}h_{3, i+2}x$, $h_{3, i+2}h_{3, i}x$ and $h_{3, i}h_{3, i+1}h_{3, i+2}x$, where $x$ is some generator of $H^*E[h_{1, i}, h_{2, i}]$.

By (3.2) and the product relations with $e_{3,i}$ in Table 3.1 we can compute the first differentials then get the generators of the $E_2$-term. And each generator is the lead term of  a cocycle in $E[h_{3, i}, h_{2, i}, h_{1, i} | i \in \mathbb{Z}/3]$. All of the cocycles determined by the generators of the $E_2$-term are the generators of the complex. With isomorphic classes  and base change we get  the generators of $H^*E[h_{3, i}, h_{2, i}, h_{1, i} | i \in \mathbb{Z}/3]$ as follows.

\begin{align*}
 & dim0: &&1, &&\qquad  \\
  & dim1: && \rho \in E_2^{1, 0, 5 }, && \qquad h_{1, i} \in E_2^{1, qp^i, 1},\\
  &dim2: && \rho h_{1, i}, \,\,  e_{4, i} \in E_2^{2, qp^i, 6}, &&\qquad g_i \in E_2^{2, q(p^{i+1} + 2p^i), 4},\\
  &  && k_i \in E_2^{2, q(2p^{i+1} + p^i), 4},&\\
 &dim3: &&  \rho e_{4, i} \in E_2^{3, qp^i, 11}, &&   \rho g_i, \,\,  \mu_i \in E_2^{3, q(p^{i+1} + 2p^i), 9},\\
 &  && \rho k_{i+1}, \, \,  \nu_i  \in E_2^{3, q(2p^{i+2} + p^{i+1}), 9}, && \xi \in E_2^{3, 0, 9},\\
 & &&  e_{4, i}h_{1, i} \in E_2^{3, 2qp^{i}, 7}, && e_{4, i}h_{1, i+1} \in E_2^{3, q(p^{i+1} + p^i), 7},\\
 & & & g_ih_{1,i+1} \in E_2^{3, 2q(p^{i+1} + p^i), 5};\\
 & dim4: && \rho \mu_i \in E_2^{4, q(p^{i+1} + 2p^i), 14}, &&\rho \nu_i \in E_2^{4, q(2p^{i+2} + p^{i+1}), 14},\\
 &       && \rho \xi \in E_2^{4, 0, 14},                         & &\rho e_{4, i}h_{1, i+1}, \, \, e_{4, i}e_{4, i+1} \in E_2^{4, q(p^{i+1} + p^i), 12},\\
  &      & &\rho e_{4, i}h_{1, i}, \, \, e_{4, i}^2, \, \, \theta_i\in E_2^{4, 2qp^{i}, 12}, && \rho g_i h_{1,i+1}, \, \, e_{4, i}k_i, \, \, e_{4, i+1}g_{i} \in E_2^{4, 2q(p^{i+1} + p^i), 10},\\
 & &&e_{4, i}g_{i+1} \in E_2^{4, qp^{i+1}, 10};\\
  & & &\\
 & dim5: &&\rho e_{4, i} e_{4, i+1} \in E_2^{5, q(p^{i+1} + p^i), 17}, &&\rho \theta_i, \,\, \rho e_{4, i}^2, \,\, \eta_i \in E_2^{5, 2qp^{i}, 17},\\
 &  &&e_{4, i+1} \mu_{i}, \, \, \rho e_{4, i+1}g_{i}, \, \, \rho e_{4, i}k_i \in E_2^{5, 2q(p^{i+1} + p^i), 15}, &&e_{4, i} \nu_i, \, \, \rho e_{4, i+1}g_{i+2} \in E_2^{5, qp^{i+2}, 15}, \\
 &  &&e_{4, i}^2h_{1, i+1} \in E_2^{5, q(p^{i+1} + 2p^i), 13}, && e_{4, i}^2h_{1, i+2} \in E_2^{5, q(  2p^i + p^{i+2}), 13},\\
 & &&e_{4, i}e_{4, i+1}h_{1, i+2}=e_{4, 0}e_{4, 1}h_{1, 2}  \in E_2^{5, 0, 13}; &&\\
 & dim6: &&\rho \eta_i \in E_2^{6, 2qp^{i}, 22}, &&\rho e_{4, i} \mu_{i+2} \in E_2^{6, 2q( p^i + p^{i+2} ), 20},\\
 &       &&\rho e_{4, i} \nu_i \in E_2^{6, qp^i, 20}, && \rho e_{4, i}e_{4, i+1}h_{1, i+2} \in E_2^{6, 0, 18},\\
 &       &&\rho e_{4, i}^2h_{1, i+1}, \, \, e_{4, i}^2e_{4, i+1} \in E_2^{6, q(p^{i+1} + 2p^i), 18}, &&e_{4, i}^2e_{4, i+2}, \, \, \rho e_{4, i}^2h_{1, i+2}  \in E_2^{6, q( 2p^i + p^{i+2} ), 18},\\
 &        && e_{4, i}e_{4, i+1} g_{i+2} \in E_2^{6, q(  p^i + p^{i+2}), 16};\\
  & dim7: && \rho e_{4, i}^2e_{4, i+2} \in E_2^{7, q(2p^{i} + p^{i+2}), 23}, && \rho e_{4, i}^2e_{4, i+1} \in E_2^{7, q(p^{i+1} + 2p^i), 23},\\
 &        &&e_{4, i+1}e_{4, i+2} \mu_{i}, \, \, \rho e_{4, i}e_{4, i+1} g_{i+2} \in E_2^{7, q(p^{i+1} + p^i), 21};\\
 & dim8:   && \rho e_{4, i}e_{4, i+1} \mu_{i+2}\in E_2^{8, q(p^i + p^{i+2}), 26}, &&e_{4, i}^2e_{4, i+2}g_{i+1}=e_{4, 0}^2e_{4, 2}g_{1} \in E_2^{8, 0, 22};\\
 & dim9:  && \rho e_{4, i}^2e_{4, i+2}g_{i+1} \in E_2^{9, 0, 27}.
\end{align*}\\
 where $ \rho:=\sum h_{3, i}$, $e_{4, i}=h_{3, i}h_{1, i} + h_{2, i}h_{2, i+2} + h_{1, i}h_{3, i+1}$, $\xi=\sum h_{3, i+1}e_{3, i} + \sum h_{2, i}h_{2, i+1}h_{2, i+2}$, $\mu_i=h_{3,i}h_{2,i}h_{1,i}$, $\nu_i=h_{3,i}h_{2,i+1}h_{1,i+2}$, $\theta_i=h_{3, i}h_{2, i+2}h_{2, i}h_{1, i} $, $\eta_i= h_{3, i}h_{3, i+1}h_{2, i+2}h_{2, i}h_{1, i} $.

From the May filtration of the generators in $E_2^{*, *, * } = H^*E[h_{3, i}, h_{2, i}, h_{1, i}]$, one can easily see that the May spectral sequence $\{E_r^{s, t, M}, d_r \}\Rightarrow H^*S(3)$ collapses at $E_2$-term for each generator $h \in E_2^{s, *, M} \xrightarrow{d_r} E_2^{s+1, *, M-r}=0$. Thus we get the  $ \mathbb{Z}/p$-module  $H^*S(3)$.

\begin{prop} [\cite{Ya} Theorem 4.2]
 $H^*S(3)$  is isomorphic to $E[ \rho] \otimes M$, where $M$ is a $ \mathbb{Z}/p$-module generated by the following listed elements:

\begin{enumerate}[\notag]
\item dim0:      1; \\

\item  dim1: $h_{1, i}$; \\

\item dim2: $e_{4, i}$,\quad
            $g_i$, \quad
            $k_i$;\\

 \item dim3:   $e_{4, i}h_{1, i}$,\quad
                $e_{4, i}h_{1, i+1}$, \quad
                $ g_ih_{1,i+1}$,\quad
                $\mu_i$,\quad
                $\nu_i$, \quad
                $ \xi $;\\

\item dim4: $e_{4, i}^2$, \quad
           $e_{4, i}e_{4, i+1}$,\quad
           $e_{4, i}g_{i+1}$, \quad
           $e_{4, i}g_{i+2}$, \quad
           $e_{4, i}k_i$, \quad
            $\theta_i$;\\

 \item dim5:  $e_{4, i}^2h_{1, i+1}$,\quad
              $e_{4, i}^2h_{1, i+2}$, \quad
              $e_{4, i}e_{4, i+1}h_{1, i+2}$,\quad
              $e_{4, i} \mu_{i+2}$, \quad
              $e_{4, i} \nu_i$, \quad
               $\eta_i$, \, $(e_{4, i}e_{4, i+1}h_{1, i+2} = e_{4, i+1}e_{4, i+2}h_{1, i} )$;\\

\item dim6:  $e_{4, i}^2e_{4, i+1}$, \quad
             $ e_{4, i}^2e_{4, i+2}$, \quad
              $e_{4, i}e_{4, i+1}g_{i+2}$; \\

\item dim7:   $e_{4, i}e_{4, i+1} \mu_{i+2}$; \\

\item dim8:    $e_{4, i}^2e_{4, i+2}g_{i+1}$,  \quad $(e_{4, i}^2e_{4, i+2}g_{i+1}=e_{4, i+1}^2e_{4, i}g_{i+2})$.\\

\end{enumerate}

\end{prop}

  Also by the relation among cohomology degrees, inner degrees and May filtrations, we know that as a ring, $H^*S(3)\cong H^*E[h_{3, i}, h_{2, i}, h_{1, i}] $.
  Therefore, we are able to determine the ring structure of $H^*S(3)$.

Summarizing the results above, we have the following
\begin{theorem}[\cite{Ya} Proposition 4.3, Theorem 4.4]
 The $\mathbb{Z}/p$-algebra $H^*S(3)$ is generated by the elements $\{  h_{1,i}, g_i, k_i, e_{4, i}, \mu_i, \nu_i, \xi, \theta_i, \eta_i, \rho |i \in \mathbb{Z}/3 \}$ satisfying the product relations given in the appendix. Its Poincar{\'e} series is $(1+t)^3(1 + t + 6t^2 + 3t^3 + 6t^4 + t^5 + t^6)$.
\end{theorem}

\bigskip

\section{A nontrivial product in stable homotopy groups on spheres}
\bigskip

In this section, we  turn to  the nontrivial products in stable homotopy groups on spheres as an application of the algebraic structure of the cohomology of the Morava stabilizer algebra $S(3)$ in the Adams-Novikov spectral sequence.

The canonical homomorphism $BP_* \to v_3^{-1}BP_*/I_3$ induces a homomorphism
$$\varphi: \operatorname{Ext}_{BP_{\ast}BP}(BP_{\ast}, BP_{\ast}) \to \operatorname{Ext}_{BP_*BP}^{*,*}(BP_*, v_3^{-1}BP_*/I_3)\cong
\operatorname{Ext}_{\Sigma(3)}(K(3)_*, K(3)_*).
$$
Specifically, by  \cite{Mi2},  $\varphi$ is induced by  the reduction map  from cobar complex $C^*_{BP_*BP}BP_*$ to complex $C_{\Sigma(3)}$, where $C^s_{BP_*BP}BP_* = \overline{BP_*BP}^{\otimes s} \bigotimes_{BP_*} BP_*$,  $\overline{BP_*BP}=ker\varepsilon$, and $\varepsilon$ is the counit of Hopf algebroid $(BP_*, BP_*BP )$.

 In the cobar complex $C_{\Sigma(3)}$, we have $d(v_3)=0$. In other words, the differential $d$ is $v_3$-linear.
Furthermore, since we have
\begin{eqnarray}
\operatorname{Ext^*}_{\Sigma(3)}(K(3)_*, K(3)_*) = H^*S(3) \otimes_{\mathbb{Z}/p} K(3)_*.
\end{eqnarray}
we may set $v_3=1$ for the sake of simplicity, if we allow ourselves to consider non-homogeneous elements. The $v_3$-linear property of $d$ ensures that the computation won't be any different.

Recall
$$  \varphi (\alpha_1)  =h_{1,0} \quad \text{and} \quad  \varphi(\beta_1) = -b_{1, 0},  $$
which are shown by Ryo Kato and Katsumi Shimomura (\cite{Ka}).  Following their work we have the following:
\begin{lemma} \label{Thom map}
Let $p \geq7$ be a prime number.
\begin{enumerate}
\item For any integers $n \geq 0$ and  $s =\frac{p^{(2i-1)}+1}{p+1}, i \geq 1$, we have
  \begin{equation}
 \varphi (\beta_{sp^n/p^n})=
\begin{cases}
 -b_{1, n}, \qquad i=1, \\
0,  \quad \, \, \, \qquad  i>1
\end{cases}
\end{equation}
where $\beta_{sp^k/j}$ is defined as in \cite{Ra}.

\item  For any integer $ s  >0$,
 $$\varphi(\gamma_s)  =  s(s^2-1)\nu_0 -s(s-1)\rho k_1.$$
\end{enumerate}
\end{lemma}

\begin{proof}
Part 1 is immediate from the Lemma 6.42 of \cite{Ra}.\\
 Part 2 has been appeared in \cite{Ka}, but we want make it more clear.

  In the cobar complex $C^*_{BP_*BP}BP_* $, by (2.1), (2.2) and (2.3), we get
 $$
   d(v_3^s) \equiv sv_2v_3^{s-1}t_1^{p^2} + \binom{s}{2} v_2^2v_3^{s-2}t_1^{2p^2}+ \binom{s}{3} v_2^3v_3^{s-3}t_1^{3p^2}  \mod {(p, v_1,  v_2^4)},
  $$
   which imply
     $$\delta_2 (v_3^s)   \equiv sv_3^{s-1}t_1^{p^2} + \binom{s}{2} v_2v_3^{s-2}t_1^{2p^2}+ \binom{s}{3} v_2^2v_3^{s-3}t_1^{3p^2} \text{mod $ v_2^3$}.$$

Recall $d(t_1^{p^{n+1}}) = -pb_{1, n}$, and we obtain
     \begin{eqnarray*}
       \delta_1 \delta_2 (v_3^s) & \equiv& s(s-1) v_3^{s-2} t_2^p \otimes t_1^{p^2} + \binom{s}{2}v_3^{s-2}t_1^p \otimes t_1^{2p^2} + s(s-1)(s-2) v_3^{s-3}v_2 t_1^{p^2}t_1^p \otimes t_1^{p^2} \\
      & &+ s\binom{s-1}{2} v_3^{s-3}v_1t_1^{2p^2} \otimes t_1^{p^2} +  s\binom{s-1}{2} v_3^{s-3}v_2t_1^{p^2}t_1^p \otimes t_1^{2p^2} +  s\binom{s-1}{2} v_3^{s-3}v_2t_2^p \otimes t_1^{2p^2} \\
      &&+   s\binom{s-1}{2} v_3^{s-3}v_1t_2^pt_1^p \otimes t_1^{2p^2} \text{mod $ (v_1, v_2)^2$}.
\end{eqnarray*}

 Notice that $ d(v_2)\equiv pt_2 $ mod $(p^2, v_1)$,
 $d(v_3)\equiv pt_3 $ mod $(p^2, v_1, v_2)$,
 and $d(t_2^p)=-t_1^p \otimes t_1^{p^2}+v_1^pb_{1, 1}-pb_{2, 0}$, in the complex $C^*_{BP_*BP}BP_* $, then we have
\begin{align*}
 \delta_0 \delta_1 \delta_2 (v_3^s) \equiv & s(s-1)v_3^{s-2}(-b_{2, 0}t_1^{p^2}+ t_2^pb_{1,1})\\
       & +\frac{s(s-1)}{2}v_3^{s-2}(-b_{1, 0} \otimes t_1^{2p^2} + 2t_1^p \otimes b_{1, 1} (1 \otimes t_1^{p^2} + t_1^{p^2} \otimes 1)) \\
      & + s(s-1)(s-2)v_3^{s-3} t_3 \otimes t_2^p \otimes t_1^{p^2}
       + \cdots  \, \,  \text{mod $(p, v_1, v_2)$.}
 \end{align*}
So \begin{eqnarray*}\varphi (\gamma_s) &=& - \underline{ s(s-1)b_{2, 0}t_1^{p^2}}_{8}
                                        + \underline{s(s-1)t_2^p \otimes b_{1, 1} }_{9}
                                        -\underline{\frac{1}{2}s(s-1)b_{1, 0} \otimes t_1^{2p^2} }_{6}\\
                                  & &  +\underline{s(s-1) t_1^p \otimes b_{1, 1}( 1 \otimes t_1^{p^2} + t_1^{p^2} \otimes 1 )}_{2}+ s(s-1)(s-2) t_3 \otimes t_2^p \otimes t_1^{p^2} + \cdots
\end{eqnarray*}

 But in the cobar complex $C_{k(3)_*K(3)}^* K(3)$ we have:
\begin{eqnarray*}
   d (t_5^p \otimes t_1^{p^2}) &=& \underline{ -t_1^p \otimes t_4^{p^2} \otimes t_1^{p^2} }_{1}
                                 - \underline{ t_2^p \otimes t_3 \otimes t_1^{p^2}}_{5}
                                 -t_3^p \otimes t_2^p \otimes t_1^{p^2}
                                  - \underline{ t_4^p \otimes t_1^{p^2} \otimes t_1^{p^2}}_{4}
                                  + \underline{b_{2, 0} \otimes t_1^{p^2} }_{8}.\\
   d (t_2^p \otimes t_4^{p^2}) &=& -\underline{t_1^p \otimes t_1^{p^2} \otimes t_4^{p^2}}_{1}
                                 + \underline{t_2^p \otimes t_1^{p^2} \otimes t_3}_{3}
                                 + t_2^p \otimes t_2^{p^2} \otimes t_2^p
                                 + \underline{t_2^p \otimes t_3^{p^2} \otimes t_1^{p^2}}_{7}
                                 - \underline{t_2^p \otimes b_{1, 1}}_{9}. \\
  d (t_1^p \otimes t_1^{p^2}t_4^{p^2}) &=& \underline{ t_1^p \otimes t_4^{p^2} \otimes t_1^{p^2}}_{1}
                                     - \underline{ t_1^p \otimes b_{1, 1}(1 \otimes t_1^{p^2} + t_1^{p^2} \otimes 1)}_{2} + \underline{ t_1^p \otimes t_1^{p^2} \otimes t_4^{p^2}}_{1} + \cdots\\
                               -d (t_2^p \otimes t_1^{p^2}t_3) &=& -\underline{t_2^p \otimes t_1^{p^2} \otimes t_3 }_{3}
                                  - \underline{t_2^p \otimes t_3 \otimes t_1^{p^2} }_{5} + \cdots\\
-2d (t_2^p t_3\otimes t_1^{p^2}) &=& +2t_3 \otimes t_2^p\otimes t_1^{p^2}
                                   + 2\underline{t_2^p \otimes t_3 \otimes t_1^{p^2} }_{5} + \cdots\\
\frac{1}{2}d (t_4^p \otimes t_1^{2p^2}) & =& \frac{1}{2}\underline{ b_{1, 0} \otimes t_1^{2p^2} }_{6}
                                              + \underline{t_4^p \otimes t_1^{p^2} \otimes t_1^{p^2} }_{4} +\cdots\\
d (t_2t_3^{p^2} \otimes t_1^p) &=& -t_3^{p^2} \otimes t_2^p \otimes t_1^p
                                  -\underline{t_2^p \otimes t_3^{p^2} \otimes t_1^p }_{7} + \cdots
\end{eqnarray*}
Thus $\varphi (\gamma_s) = s(s^2-1) t_3 \otimes t_2^p \otimes t_1^{p^2} - s(s-1) \rho \otimes t_2^p \otimes t_1^{p^2} + s(s-1)t_2^p \otimes t_2^{p^2} \otimes t_2^p + \cdots = s(s^2-1) \nu_0 - s(s-1) \rho k_1$, for the monomials with same tabs will disappear.

\end{proof}

Now we can prove our main result.
\begin{proof}[Proof of Theorem \cref{nontrivprod}]
By Cohen \cite{Co}, $\zeta_n$ is represented by $\alpha_1 \beta_{p^n/p^n} + \alpha_1 x \in Ext_{BP_*BP}^{*,* }(BP_*, BP_*)$, which is the $ E_2$-term of the Adams-Novikov spectral sequence, where $ x = \sum\limits_{s,k, j } a_{s, k, j}\beta_{sp^k/j}$ and $a_{1, n, p^n} =0$. Comparing the inner degrees, we get
  $$ 2(p^2-1)sp^k - 2(p-1)j = 2(p^2-1)p^n - 2(p-1)p^n. $$
  That is, $ p^k(sp +s ) = p^{n+1} +j $. And by the theorem 2.6 of \cite{Mi},  $j \leq p^k + p^{k-1} -1$, we get $k \leq n$ and $ j=p^k $. Thus
  $$  x = \sum\limits_{n+1-k \, \\ \text{ odd}} a_{s, k, p^k} \beta_{sp^k/p^k} , $$
  where $s=\frac{p^{n-k +1} +1 }{p+1} >1$.

For $ (\alpha_1\beta_{p^n/p^n} +\alpha_1x) \cdot \gamma_s \in E_2^{6, *}$, by the Lemma \cref{Thom map} ,we have
\begin{eqnarray*}
  \varphi (\alpha_1(\beta_{p^n/p^n} +x) \cdot \gamma_s) &=&  -s(s^2-1) h_{1, 0}\cdot b_{1, \overline{n}}\cdot \nu_0 + s(s-1) h_{1, 0}\cdot b_{1, \overline{n}}\cdot \rho k_1\\
                                                     &=& -s(s^2-1) h_{1, 0}\cdot e_{4, \overline{n+1}}\cdot \nu_0 \\
                                                    &\equiv& \begin{cases}
\frac{ s(s^2-1)}{3}e_{4, 2}e_{4, 0}g_1 \neq 0& \text{if  $\overline{n} \equiv 0$, $s\not \equiv 0$, $\pm 1$  mod $(p) $},\\
   0 & \text{if $\overline{n} \equiv 1$ mod $(p)$},\\
 \frac{ s(s^2-1)}{3}e_{4, 0}e_{4, 1}g_2 \neq 0& \text{if $\overline{n} \equiv 2$, $s \not \equiv 0$, $\pm 1$  mod $(p) $},\\
\end{cases}
\end{eqnarray*}
here, $\overline{n}$ is mod (3) reduction of n.
Thus  $ (\alpha_1\beta_{p^n/p^n} + \alpha_1 x) \cdot \gamma_s \neq 0$ under the conditions in the theorem.

Notice that the inner degrees of elements  in $\operatorname{Ext}_{BP_*BP}^*(BP_*, BP_*)$  are divisible by $q$, where $q=2(p-1)$. This means that the first nontrivial differential may be $d_{q+1}$, so $\alpha_1(\beta_{p^n/p^n} +x) \cdot \gamma_s \in E_2^{6, *}$ may not be killed be any differentials, and we conclude.
\end{proof}
\bigskip

\appendix
\section{The list of product relations of any two generators of $H^*S(3)$}

\bigskip

Here we lose the products that equal to zero and the  proof which is trivial but tedious.  \\

dim3:
\begin{align*}
& e_{4, i} \cdot h_{1, i+2}=e_{4, i+2}h_{1, i}, &  &k_i \cdot h_{1, i} = -g_i h_{1, i+1},\\
\end{align*}

dim4:
\begin{align*}
   &e_{4, i} \cdot k_{i+1}= e_{4, i+1}g_{i+2},   &&\mu_i \cdot h_{1, i+2}= -\frac{1}{3}e_{4, i+2}g_{i},\\
    & \mu_i \cdot h_{1, i+1} = \frac{1}{3} e_{4, i +1}g_i - \frac{2}{3}e_{4, i}k_i + \frac{1}{3}\rho g_i h_{1, i+1},&&\nu_i \cdot h_{1, i} =\frac{1}{3}e_{4, i+1}g_{i+2}, \\
   &  \nu_i \cdot h_{1, i+1} =  \frac{2}{3}e_{4, i +2}g_{i+1} - \frac{1}{3} e_{4, i+1}k_{i+1} - \frac{1}{3} \rho g_{i+1} h_{1, i+2}, & &\xi  \cdot h_{1, i} = -e_{4, i+2}g_{i},\\
 \end{align*}
 dim5:

   \begin{align*}
  & e_{4, i}e_{4, i+1} \cdot h_{1, i} =  e_{4, i}^2h_{1, i+1},   &&e_{4, i}e_{4, i+1} \cdot h_{1, i+1} = e_{4, i+1}^2h_{1, i},\\
  &  \theta_i \cdot  h_{1, i+2} = -\frac{1}{2} e_{4, i}^2h_{1, i+2},        &&e_{4, i}h_{1, i} \cdot e_{4, i+1} = e_{4, i}^2h_{1, i+1}, \\
  & e_{4, i}h_{1, i} \cdot e_{4, i+2} = e_{4, i}^2h_{1, i+2},     &&e_{4, i}h_{1, i+1} \cdot e_{4, i+1}= e_{4, i+1}^2h_{1, i},\\
  & e_{4, i}h_{1, i+1} \cdot e_{4, i+2}= e_{4, i+2}e_{4, i}h_{1, i+1}, &&e_{4, i} \cdot \mu_{i+1} =  \frac{2}{3}\rho e_{4, i}g_{i+1} - e_{4, i+2}v_{i+2},\\
  &   \mu_i  \cdot g_{i+1}  = \frac{1}{2}e_{4, i+1}^2h_{1, i},           && \mu_i  \cdot g_{i+2} = -\frac{1}{2} e_{4, i}^2h_{1, i+2},\\
  &  \mu_i  \cdot k_{i+1}  = \frac{1}{6}e_{4, i}e_{4, +1}h_{1, i+2}, &&\nu_i  \cdot e_{4, i+1}  = -e_{4, i+2} \mu_{i+1} + \frac{1}{3} \rho e_{4, i+2}g_{i+1} + \frac{1}{3}\rho e_{4, i+1}k_{i+1},  \\
   &\nu_i  \cdot g_i =  \frac{1}{6}e_{4, i}e_{4, +1}h_{1, i+2},  &&\nu_i  \cdot k_i= \frac{1}{2}e_{4, i+1}^2h_{1, i+2},\\
  &\nu_ik_{i+2} = -\frac{1}{2}e_{4, i+2}^2h_{1, i},            && \xi   \cdot e_{4, i} = \rho e_{4, i+2}g_i-3e_{4, i+1}v_{i+1},\\
  &\xi  \cdot g_i= -\frac{1}{2}e_{4, i}^2h_{1, i+1},                       && \xi \cdot k_i = \frac{1}{2}e_{4, i+1}^2h_{1, i},\\
 \end{align*}
 dim6:
 \begin{align*}
 &  e_{4, i}h_{1, i} \cdot \mu_{i+1} =\frac{1}{3}e_{4, i+1}e_{4, i+2}g_i, && e_{4, i}h_{1, i} \cdot \nu_i = -\frac{1}{3}e_{4, i}e_{4, i+1}g_{i+2},\\
 &e_{4, i}h_{1, i+1} \cdot \mu_{i+2} = \frac{1}{3}e_{4, i+1}e_{4, i}g_{i+2} , &&e_{4, i}h_{1, i+1} \cdot \nu_i  = -\frac{1}{3}  e_{4, i +2}e_{4, i}g_{i+1},\\
 & \mu_i \cdot \mu_{i+1} =  -\frac{1}{3}\rho e_{4, i+1}^2h_{1, i} -\frac{1}{6} e_{4, i+1}^2 e_{4, i}, &&  \mu_i \cdot  \xi = \frac{1}{6}\rho e_{4, i}^2h_{1, i+1} + \frac{1}{6} e_{4, i}^2e_{4, i+1},\\
  & \nu_i \cdot \nu_{i+1} =\frac{1}{3} \rho e_{4, i+2}^2h_{1, i} - \frac{1}{6} e_{4, i+2}^2e_{4, i}, && \nu_i \cdot \xi  = \frac{1}{6}e_{4, i+2}^2 e_{4, i+1}-\frac{1}{6}\rho e_{4, i+2}^2h_{1, i+1},\\
    & e_{4, i}k_i \cdot e_{4, i+2} =  e_{4, i+1}e_{4, i+2}g_i,  &&e_{4, i}^2 \cdot g_{i+1} = e_{4, i+1}e_{4, i+2}g_i, \\
 & e_{4, i}^2 \cdot k_{i+1} = e_{4, i}e_{4, i+1}g_{i+2}, &&e_{4, i}e_{4, i+1} \cdot k_{i+1}  = e_{4, i+2}e_{4, i}g_{i+1},\\
 & e_{4, i}g_{i+1} \cdot e_{4, i} =  e_{4, i+1}e_{4, i+2}g_i,  && \theta_i \cdot e_{4, i +1} = \frac{1}{3}\rho e_{4, i}^2h_{1, i+1} + \frac{1}{3}e_{4, i}^2e_{4, i+1},\\
 &\theta_i \cdot e_{4, i+2} = -\frac{1}{3}\rho e_{4, i}^2h_{1, i+2} -\frac{1}{6} e_{4, i}^2 e_{4, i+2}, &&\theta_i \cdot g_{i+1} =  -\frac{1}{6}e_{4, i+1}e_{4, i+2}g_i,\\
 &\theta_i \cdot k_{i+1} = \frac{1}{3}e_{4, i}e_{4, i+1}g_{i+2}, &&e_{4, i} \mu_{i+2} \cdot h_{1, i+1} =-\frac{1}{3}e_{4, i}e_{4, i+1}g_{i+2},\\
 &e_{4, i} \nu_i \cdot h_{1, i} = \frac{1}{3}e_{4, i} e_{4, i+1}g_{i+2}, &&e_{4, i} \nu_i \cdot h_{1, i+1} =  \frac{1}{3} e_{4, i +2} e_{4, i}g_{i+1},\\
  & \eta_i \cdot h_{1, i+1} = \frac{1}{6}\rho e_{4, i}^2h_{1, i+1} +\frac{1}{6} e_{4, i}^2e_{4, i+1}, &&\eta_i \cdot h_{1, i+2} =\frac{1}{6} \rho e_{4, i}^2h_{1, i+2} - \frac{1}{6}e_{4, i}^2 e_{4, i+2},\\
  \end{align*}
 dim7:
 \begin{align*}
 & e_{4, i}^2 \cdot  \mu_{i+1} = e_{4, i+1}e_{4, i+2} \mu_{i},  && e_{4, i}^2 \cdot  \nu_i = \frac{2}{3}\rho e_{4, i}e_{4, i+1}g_{i+2} - e_{4, i}e_{4, i+1}\mu_{i+2},\\
 & e_{4, i} e_{4, i+1} \cdot  \nu_i =-e_{4, i}e_{4, i+1} \mu_{i+2} +  \frac{2}{3} \rho e_{4, i+2}e_{4, i} g_{i+1},  &&e_{4, i} e_{4, i+1} \cdot  \xi   = -\rho e_{4, i+1}e_{4, i+2}g_{i} + 3e_{4, i+1}e_{4, i+2}\mu_{i},\\
 & \theta_i \cdot \mu_{i+1} = \frac{1}{2} e_{4, i+1}e_{4, i+2} \mu_i, &&e_{4, i} \nu_i \cdot e_{4, i} =\frac{2}{3}\rho e_{4, i}e_{4, i+1}g_{i+2}- e_{4, i}e_{4, i+1}\mu_{i+2},\\
 & e_{4, i} \nu_i \cdot e_{4, i+1} = -e_{4, i+2}  e_{4, i} \mu_{i+1} +\frac{2}{3}\rho e_{4, i+2}  e_{4, i}g_{i+2},  && \eta_i \cdot e_{4, i+1} =\frac{1}{6} \rho e_{4, i}^2 e_{4, i+1},\\
 & \eta_i \cdot e_{4, i+2} = \frac{1}{6} \rho e_{4, i+1}^2 e_{4, i},   && \eta_i \cdot g_{i+1}=  \frac{1}{2} e_{4, i+1}e_{4, i+2} \mu_i ,\\
 &\eta_i \cdot k_{i+1} =  - \frac{1}{2} e_{4, i}e_{4, i+1}\mu_{i+2} +\frac{1}{3} \rho e_{4, i}e_{4, i+1}g_{i+2},\\
 \end{align*}
 dim8:
 \begin{align*}
 &  e_{4, i}^2 \cdot e_{4, i+1}k_{i+1} = e_{4, i}^2 e_{4, i+2}g_{i+1},        &&e_{4, i}^2h_{1, i+1} \cdot  \nu_i = -\frac{1}{3}e_{4, i}^2e_{4, i+2}g_{i+1},\\
  & e_{4, i}e_{4, i+1}h_{1, i+2} \cdot  \xi = e_{4, i+1}^2e_{4, i}g_{i+2}, &&e_{4, i} \mu_{i+2} \cdot e_{4, i+1}h_{1, i+1} =-\frac{1}{3} e_{4, i}^2 e_{4, i+2}g_{i+1},\\
  \end{align*}
   \begin{align*}
  &e_{4, i}  \nu_i \cdot e_{4, i}h_{1, i+1}=\frac{1}{3}e_{4, i}^2e_{4, i+2}g_{i+1},  &&\eta_i \cdot g_{i+1}h_{1,i+2} =-\frac{1}{6}e_{4, i}^2e_{4, i+2}g_{i+1},\\
 &e_{4, i}^2e_{4, i+1} \cdot   k_{i+1} =e_{4, i}^2e_{4, i+2}g_{i+1},  &&e_{4, i}e_{4, i+1} \mu_{i+2} \cdot h_{1, i+1} =-\frac{1}{3}e_{4, i+1}^2e_{4, i}g_{i+2},
  \end{align*}
 dim9:
 \begin{align*}
 &e_{4, i} \mu_{i+2} \cdot e_{4, i+1}^2 =\frac{1}{3} \rho e_{4, i}^2e_{4, i+2}g_{i+1}, &&e_{4, i} \mu_{i+2} \cdot \theta_{i+1}  =\frac{1}{6} \rho e_{4, i}^2e_{4, i+2}g_{i+1},\\
 &e_{4, i}  \nu_i \cdot e_{4, i}e_{4, i+1} =\frac{1}{3}  \rho e_{4, i+2}e_{4, i}^2g_{i+1}, &&\eta_i \cdot e_{4, i+2}g_{i+1} = \frac{1}{6}\rho e_{4, i}^2e_{4, i+2}g_{i+1},\\
 & \eta_i \cdot e_{4, i+1}k_{i+1} = \frac{1}{6} \rho e_{4, i}^2e_{4, i+2}g_{i+1},  &&e_{4, i}^2e_{4, i+1} \cdot  \nu_i =\frac{1}{6}\rho e_{4, i}^2 e_{4, i+2}g_{i+1},\\
 &e_{4, i}^2e_{4, i+2} \cdot \mu_{i+1} =\frac{1}{3} \rho e_{4, i}^2 e_{4, i+2}g_{i+1}, && e_{4, i}e_{4, i+1} \mu_{i+2} \cdot e_{4, i+1} = \frac{1}{3}\rho e_{4, i}^2e_{4, i+2}g_{i+1}.\\
 \end{align*}

\begin{remark}
The multiplications in \cite{Ya}  are corresponded with above, except  $a_0g'_0=h_0b'_0-h_1b_0$. From our calculations, it should be  $a_0g'_0=h_0b'_0-2h_1b_0$.
\end{remark}

\begin{center}
 \vskip 20pt
{\bf Gu Xing}\\
{\it The Max Planck Institute for Mathematics, Vivatsgasse 7, 53111, Bonn, Germany.}\\
{\tt gux2006@mpim-bonn.mpg.de}\\  \vskip 10pt
{\bf Wang Xiangjun}\\
{\it School of Mathematics, Nankai University,
Tianjin 300071, P.R. China}\\
{\tt xjwang@nankai.edu.cn}\\  \vskip 10pt
{\bf Wu Jianqiu}\\
{\it School of Mathematics, Nankai University,
Tianjin 300071, P.R. China}\\
{\tt wujianqiu@mail.nankai.edu.cn}\\
\end{center}

\end{document}